\documentclass{amsart}
\newtheorem{theorem}{Theorem}[section]

\usepackage{graphicx}
\theoremstyle{definition}
\newtheorem{definition}[theorem]{Definition}

\theoremstyle{remark}

\numberwithin{equation}{section} 
\begin{document}
\title{Modular Logarithm Unequal}
\author{Wu Sheng-Ping 
}
\email{hiyaho@126.com}
\address{
Tianmen, Hubei Province, The People's Republic of
China. Postcode: 431700 }
\date{Jan 27, 2020}

\begin{abstract} The main idea of this article is simply calculating integer functions
in module. The algebraic in the integer modules is studied in completely new
style. By a careful construction, a result is proven that two finite numbers are with
unequal logarithms in a corresponding module, and is applied
to solving a kind of high degree diophantine equation.
\end{abstract}
\maketitle
In this paper, $p$ is prime, $C$ means a constant. All numbers that are
indicated by Latin letters are integers unless with further indication.
\section{Function in module}
\begin{theorem} Define the congruence class \cite{p1} in the form:
\[
[a]_q:=[a+kq]_q,\forall k\in\mathbf{Z}
\]
\[
[a=b]_q: [a]_q=[b]_q
\]
\[
[a]_q[b]_{q'}:=[x]_{qq'}: [x=a]_q,[x=b]_{q'},(q,q') = 1
\]
then
\[
[a+b]_q=[a]_q+[b]_q
\]
\[
[ab]_q=[a]_q\cdot[b]_q
\]
\[
[a+c]_q[b+d]_{q'}=[a]_q[b]_{q'}+[c]_q[d]_{q'},(q,q') = 1
\]
\[
[ka]_q[kb]_{q'} = k[a]_q[b]_{q'},(q,q') = 1
\]
\end{theorem}
\begin{theorem}The integer coefficient power-analytic functions modulo $p$ are all the functions from mod
$p$ to mod $p$
\[
[x^0=1]_p
\]
\[
[f(x)=\sum^{p-1}_{n=0}f(n)(1-(x-n)^{p-1})]_p
\]
\end{theorem}
\begin{theorem}(Modular Logarithm) Define
\[
[lm_a(x):=y]_{p^{m-1}(p-1)}: [a^y=x]_{p^m}
\]
\[
[E:=\sum^{m'}_{i=0}p^i/i!]_{p^m}
\]
\[
1<<m<<m'
\]
then
\[
[E^x=\sum^{m'}_{i=0}x^ip^i/i!]_{p^m}
\]
\[
[lm_E(1-xp)=-\sum^{m'}_{i=1}(xp)^{i}/(ip)]_{p^{m-1}}
\]
\[
[Q(q)lm(1-xq)=-\sum^{m'}_{i=1}(xq)^i/i]_{q^m}
\]
\[
Q(q):=\prod_{p|q}[p]_{p^m}
\]
Define
\[
[lm(x):=lm_e(x)]_{p^{m-1}}
\]
e is the generating element in mod p and meets
\[
[e^{1-p^{m'}}=E]_{p^m}
\]
\end{theorem}
It's proven by comparing to the Taylor expansions of real exponent and logarithm (especially on the coefficients).
\begin{definition}
\[
[lm(px):=plm(x)]_{p^m}
\]
\end{definition}
\begin{definition}
\[
P(q):=\prod_{p|q}p
\]
\end{definition}
\begin{definition}
\[
{_q}[x]:=y:[x=y]_q,0\leq y<q
\]
\end{definition}

\section{Unequal Logarithms of Two Numbers}
\begin{theorem} If
\[
P(q)b+a<q
\]
\[
a>b>0
\]
\[
(a,b)=(a,q)=(b,q)=1
\]
then
\[
[lm(a)\neq lm(b)]_{q}
\]
\end{theorem}
\begin{proof}
Define
\[
r:=P(q)
\]
\[
\beta:=\prod_{p:p|q}[(a/b)^{v_p-1}]_{p^m},1<<m
\]
\[
v_p:=[p]_{p^m(p-1)}
\]
\par
Set
\[
0<x,x'<q
\]
\[
0<y,y'<qr+r
\]
\[
d:=(x-x',q^m)
\]
Consider
\[
[(x,y,x',y')=(b,a,b,a)]_{r}
\]
\[
[\beta^2a^{2}x^2-b^{2}y^2=\beta^2a^{2}x'^2-b^{2}y'^2=:(2,q)qN]_{q^2},(N,q)=1
\]
Checking the freedom and determination of $(x,y),(x',y')$, and using the Drawer Principle,
we find that there exist \emph{distinct} $(x,y),(x',y')$ satisfying the previous conditions.
\par
Presume
\[
(qr^n,p^m)||\beta -1,n\geq 0
\]
\[
(d,p^m)|q/r
\]
Make
\[
(s,t,s',t'):=(x,y,x',y')+qZ(b,\beta a,0,0)
\]
to set
\[
[\beta^2a^{2}s^2-b^{2}t^2=\beta^2a^{2}s'^2-b^{2}t'^2]_{p^m}
\]
Make
\[
(X,Y,X',Y'):=(s,t,s',t')+qZ'(s',-t',s,-t)
\]
to set
\[
[aX-bY=aX'-bY']_{p^m}
\]
hence
\[
[\beta^2a(X+X')=b(Y+Y')]_{p^m}
\]
\par
The variables of fraction $z,z'$ meet the equation
\[
[(aX+z)^2-(bY-\beta z')^2=(aX'+z')^2-(bY'-\beta z)^2]_{p^{m}}
\]
It's equivalent to
\[
[2(aX-\beta bY')z-2(aX'-\beta bY)z'+(1+\beta^2)(z^2-z'^2)+(a^2X^2-a^2X'^2)(1-\beta^2)=0]_{p^m}
\]
\[
[(1+\beta)(aX-aX')(z+z')+(1-\beta^3)(aX+aX')(z-z')+(1+\beta^2)(z^2-z'^2)
\]
\[
=-(a^2X^2-a^2X'^2)(1-\beta^2)]_{p^m}
\]
\[
[(z-z'+\frac{1+\beta}{1+\beta^2}a(X-X'))(z+z'+\frac{1-\beta^3}{1+\beta^2}(aX+aX'))=\frac{\beta(1-\beta^2)}{(1+\beta^2)^2}(a^2X^2-a^2X'^2)]_{p^m}
\]
In another way
\[
[(aX-bY+z+\beta z')(aX+bY+z-\beta z')=(aX'-bY'+\beta z+z')(aX'+bY'-\beta z+z'))]_{p^{m}}
\]
Make by choosing a valid $z-z'$
\[
[aX+bY+z-\beta z'=aX'+bY'-\beta z+z']_{p^m}
\]
then
\[
[aX-bY+z+\beta z'=aX'-bY'+\beta z+z']_{p^m}
\]
It's invalid, hence
\begin{equation}\label{1}
[x=x']_{(q,p^m)}\vee\neg(qr^n,p^m)||\beta-1
\end{equation}
\par
The case for $p=2$ is similar.
\par
If
\[
[\beta-1=0]_{p^l}
\]
then
\[
[a^{p-1}-b^{p-1}=0]_{p^l}
\]
\[
l<C
\]
\par
Furthermore
\begin{equation}\label{2}
q|\beta-1\wedge[x=x']_{q}=0
\end{equation}
because if not,
\[
[\beta ax-by=\beta ax'-by']_{q^2}
\]
\[
[ax-by=ax'-by']_{q^2}
\]
\[
|ax-by-(ax'-by')|<q^2
\]
\[
ax-by=ax'-by'
\]
therefore
\[
x-x'=0=y-y'
\]
It contradicts to the previous condition.
\par
So that with the condition \ref{1}
\[
\neg(qr^n,p^m)||\beta -1=[x=x']_{(q,p^m)}\wedge\neg(qr^n,p^m)||\beta -1\vee [x\neq x']_{(q,p^m)}
\]
Wedge with $(qr^n,p^m)|\beta -1$
\[
(qr^{n+1},p^m)|\beta -1=(qr^{n+1},p^m)|\beta -1\wedge [x=x']_{q}
\]
With the condition \ref{2}
\[
qr|\beta -1=0
\]
\end{proof}
\begin{theorem} For prime p and positive integer q the equation
$a^p+ b^p= c^q$
has no integer solution (a,b,c) such that $(a,b)=(b,c)=(a,c) =1,a,b>0$ if
$p,q>8$.
\end{theorem}
\begin{proof}Reduction to absurdity. Make logarithm on $a,b$ in mod $c^q$. The conditions are sufficient for a controversy.
\end{proof}

\end{document}